\documentclass[12pt]{article}%
\usepackage{amsmath}
\usepackage{amsfonts}
\usepackage{amssymb}
\usepackage{graphicx}%
\providecommand{\U}[1]{\protect\rule{.1in}{.1in}}
\newtheorem{theorem}{Theorem}

\newtheorem{lemma}[theorem]{Lemma}

\newtheorem{proposition}[theorem]{Proposition}
\newtheorem{remark}[theorem]{Remark}

\newenvironment{proof}[1][Proof]{\noindent\textbf{#1.} }{\ \rule{0.5em}{0.5em}}
\begin{document}

\title{Linear functionals and $\Delta$- coherent pairs of the second kind}
\author{Diego Dominici \thanks{Research Institute for Symbolic Computation, Johannes
Kepler University Linz, Altenberger Stra\ss e 69, 4040 Linz, Austria. e-mail:
ddominic@risc.uni-linz.ac.at} \thanks{Department of Mathematics, State
University of New York at New Paltz, 1 Hawk Dr., New Paltz, NY 12561-2443,
USA.}
\and Francisco Marcell{\'a}n \thanks{Departamento de Matem\'aticas, Universidad
Carlos III de Madrid, Escuela Polit\'ecnica Superior, Av. Universidad 30,
28911 Legan\'es, Spain. e-mail: pacomarc@ing.uc3m.es} }
\maketitle

\begin{abstract}
We classify all the \emph{$\Delta$-}coherent pairs of measures of the second
kind on the real line. We obtain $5$ cases, corresponding to all the families
of discrete semiclassical orthogonal polynomials of class $s\leq1.$

\end{abstract}

\emph{Keywords}: Discrete orthogonal polynomials, discrete semiclassical
functionals, discrete Sobolev inner products, coherent pairs of discrete
measures, coherent pairs of second kind for discrete measures.

\emph{2020 Mathematical Subject Classification}: 42C05 (primary), 33C45, 46E39 (secondary).

\section{Introduction}

The aim of this contribution is to provide a characterization of all pairs of
discrete measures $\{\rho_{0},\rho_{1}\}$ supported on the real line such that
the corresponding sequences of \emph{monic orthogonal polynomials} (MOPs for
short) $\{P_{n}(\rho_{0};x)\}_{n\geq0}$ and $\{P_{n}(\rho_{1};x)\}_{n\geq0}$
satisfy
\begin{equation}
P_{n}(\rho_{1};x)-\tau_{n}P_{n-1}(\rho_{1};x)=\frac{1}{n+1}\Delta P_{n+1}%
(\rho_{0};x),\quad n\geq1, \label{Eq-CPSK-D}%
\end{equation}
where $\tau_{n}\neq0$ for $n\geq1$. We will solve this problem by dealing with
a more general problem concerning the characterization of pairs of
quasi-definite linear functionals (with complex moments) such that the
corresponding sequences of MOPs satisfy (\ref{Eq-CPSK-D}).\newline

As we show in this contribution, we get special pairs of linear functionals
(in particular, positive definite linear functionals associated with positive
measures supported on infinite subsets of the real line) and the corresponding
associated sequences of orthogonal polynomials have interesting properties.
These pairs are \emph{discrete semiclassical linear functionals} of class at
most $s=1$ which have been studied in \cite{MR3227440} and \cite{MR4238531} in
the framework of a classification problem based on a hierarchy
structure.\newline

On the other hand, these sequences of orthogonal polynomials are related to
some problems in approximation theory. More precisely, the analysis of Fourier
expansions in terms of sequences of polynomials orthogonal with respect to the
Sobolev inner product
\begin{equation}
\langle f,g\rangle_{\mathfrak{S}}=\sum_{x=0}^{\infty}f(x)g(x)\rho
_{0}(x)+\lambda\sum_{x=0}^{\infty}\Delta f(x)\Delta g(x)\rho_{1}(x),
\label{Eq-S-InnerProduct}%
\end{equation}
defined by the pair of discrete measures $\{\rho_{0},\rho_{1}\}$ of class
$s\geq1$. It turns out that the monic Sobolev orthogonal polynomials
$S_{n}(\rho_{0},\rho_{1};x)$ with respect to the inner product $\langle
\,,\,\rangle_{\mathfrak{S}}$ satisfy the connection formulas
\begin{equation}%
\begin{array}
[c]{l}%
S_{n+1}(\rho_{0},\rho_{1};x)-\gamma_{n}S_{n}(\rho_{0},\rho_{1};x)=P_{n+1}%
(\rho_{0};x),\\[1.5ex]%
\Delta S_{n+1}(\rho_{0},\rho_{1};x)-\gamma_{n}\Delta S_{n}(\rho_{0},\rho
_{1};x)=(n+1)\left[  P_{n}(\rho_{1};x)-\tau_{n}P_{n-1}(\rho_{1};x)\right]  ,
\end{array}
\quad n\geq1, \label{Eq-DSK-SobolevOrth-to-NormalOrth}%
\end{equation}
with $S_{1}(\rho_{0},\rho_{1};x)=P_{1}(\rho_{0};x)$. The above connection
formulas between $\{S_{n}(\rho_{0},\rho_{1};x)\}_{n\geq0}$ and the sequences
of MOPs $\{P_{n}(\rho_{0};x)\}_{n\geq0}$ and $\{P_{n}(\rho_{1};x)\}_{n\geq0}$
yield a nice approach to the study of algebraic and analytic properties of the
polynomials $S_{n}(\rho_{0},\rho_{1};x)$.\newline

Based on historical information given below we will refer to pairs of measures
$\{\rho_{0},\rho_{1}\}$ satisfying property (\ref{Eq-CPSK-D}) as
\emph{$\Delta$-coherent pairs of measures of the second kind on the real
line}. The concept of coherence between a pair of probability measures
$\{\nu_{0},\nu_{1}\}$ supported on the real line was introduced in
\cite{MR1104157}. Indeed, a pair of probability measures supported on the real
line is said to be a coherent pair if the corresponding sequences of MOPs
$\{P_{n}(\nu_{0};x)\}_{n\geq0}$ and $\{P_{n}(\nu_{1};x)\}_{n\geq0}$, satisfy
\begin{equation}
P_{n}(\nu_{1};x)=\frac{1}{n+1}\left[  P_{n+1}^{\prime}(\nu_{0};x)-\rho
_{n}P_{n}^{\prime}(\nu_{0};x)\right]  ,\quad\rho_{n}\neq0,\quad n\geq1.
\label{Eq-CP}%
\end{equation}
It was shown in this case that the sequence of monic orthogonal polynomials
with respect to the inner product
\begin{equation}
\langle f,g\rangle_{S}=\int f(x)g(x)d\nu_{0}+\lambda\int f^{\prime
}(x)g^{\prime}(x)d\nu_{1}, \label{Eq-Sobolev}%
\end{equation}
satisfies the connection formulas
\begin{equation}%
\begin{array}
[c]{l}%
S_{n+1}(\nu_{0},\nu_{1};x)-\gamma_{n}S_{n}(\nu_{0},\nu_{1};x)=P_{n+1}(\nu
_{0};x)-\rho_{n}P_{n}(\nu_{0};x),\\[1.5ex]%
S_{n+1}^{\prime}(\nu_{0},\nu_{1};x)-\gamma_{n}S_{n}^{\prime}(\nu_{0},\nu
_{1};x)=(n+1)P_{n}(\nu_{1};x),
\end{array}
\quad n\geq1. \label{Eq-SobolevOrth-to-NormalOrth}%
\end{equation}
These formulas are very useful in the study of analytic properties of the
corresponding Sobolev orthogonal polynomials. For more information about
polynomials orthogonal with respect to Sobolev inner products see the updated
survey \cite{MR3360352}.\newline

The motivation in \cite{MR1104157} for introducing such pairs of measures was
their applications in the framework of Fourier expansions of functions with
respect to the Sobolev inner product $\langle\,,\,\rangle_{\mathfrak{S}}$. A
particular case of such Fourier series expansions based on Legendre-Sobolev
orthogonal polynomials had already been considered in \cite{MR1071973}, where
some numerical tests comparing these Legendre-Sobolev Fourier series
expansions and the ordinary Legendre Fourier series expansions are
presented.\newline

The pairs of probability measures supported on the real line with the property
that the corresponding sequences of MOPs satisfy (\ref{Eq-CP}) were completely
determined in 1997 by H. G. Meijer \cite{MR1451509}. He showed that if
$(\nu_{0},\nu_{1})$ is a coherent pair of measures on the real line, then one
of the measures must be classical (either Jacobi or Laguerre) and the other
one a rational perturbation of it. Note that the condition for classical
linear functionals to be part of a coherent pair was deduced in
\cite{MR1351148}.

Thus, what was proved in \cite{MR1451509} is more general than what is stated
above. The starting point of the classification in \cite{MR1451509} are
certain functional relations established in \cite{MR1343524} with respect to
pairs of quasi-definite linear functionals such that the corresponding
sequences of MOPs satisfy a relation like (\ref{Eq-CP}). Note that if the
coherence property (\ref{Eq-CP}) holds, then you get the connection formulas
in (\ref{Eq-SobolevOrth-to-NormalOrth}). According to \cite{MR1451509},
(\ref{Eq-SobolevOrth-to-NormalOrth}) holds when one of the measures in
$\{\nu_{0},\nu_{1}\}$ is a classical one. \newline

An extension of the concept of coherence, known in the literature as
$(1,1)$-coherence, is defined in terms of the corresponding MOPs as follows
(see \cite{MR2143522})
\begin{equation}
P_{n}(\nu_{1};x)-\tau_{n}P_{n-1}(\nu_{1};x)=\frac{1}{n+1}\left[
P_{n+1}^{\prime}(\nu_{0};x)-\xi_{n}P_{n}^{\prime}(\nu_{0};x)\right]  ,\quad
\xi_{n},\tau_{n}\neq0,\quad n\geq1,\ \ \label{Eq-CP11}%
\end{equation}

In this case, the Sobolev orthogonal polynomials associated with the inner
product $\langle\,,\,\rangle_{\mathfrak{S}}$ in (\ref{Eq-Sobolev}) satisfy the
connection formulas
\begin{equation}%
\begin{array}
[c]{l}%
S_{n+1}(\nu_{0},\nu_{1};x)-\gamma_{n}S_{n}(\nu_{0},\nu_{1};x)=P_{n+1}(\nu
_{0};x)-\xi_{n}P_{n}(\nu_{0};x),\\[1.5ex]%
S_{n+1}^{\prime}(\nu_{0},\nu_{1};x)-\gamma_{n}S_{n}^{\prime}(\nu_{0},\nu
_{1};x)=(n+1)\left[  P_{n}(\nu_{1};x)-\tau_{n}P_{n-1}(\nu_{1};x)\right]  ,
\end{array}
\quad n\geq1. \label{Eq-11-SobolevOrth-to-NormalOrth}%
\end{equation}
We would like to emphasize that the characterization of probability measures
satisfying the coherence property (\ref{Eq-CP11}) is given in \cite{MR2143522}
assuming that $\xi_{n},\tau_{n}\neq0$ for $n\geq1$. Therein you have
semiclassical linear functionals of class at most $1$. Its role in the study
of the sequences of orthogonal polynomials with respect to the Sobolev inner
product defined by a $(1,1)$ coherent pairs of measures has been emphasized in
\cite{MR4174712}. Notice that when $\tau_{n}=0$ and $\xi_{n}\neq0$ for
$n\geq1,$ we get the results on coherent pairs presented in \cite{MR1451509}.
Thus, a natural question is to analyze the case $\xi_{n}=0,$and $\tau_{n}%
\neq0$ for $n\geq0.$\newline

In \cite{Paco}, the concept of \emph{coherent pair of second kind} is
introduced. A pair of probability measures $\{\nu_{0},\nu_{1}\}$ supported on
the real line is said to be a coherent pair of second kind if the
corresponding sequences of MOPs $\{P_{n}(\nu_{0};x)\}_{n\geq0}$ and
$\{P_{n}(\nu_{1};x)\}_{n\geq0}$ satisfy%

\begin{equation}
\frac{1}{n+1}P_{n+1}^{\prime}(\nu_{0};x)=P_{n}(\nu_{1};x)-\tau_{n}P_{n-1}%
(\nu_{1};x)\quad n\geq1, \label{Eq-CPSK}%
\end{equation}
where $\tau_{n}\neq0$ for $n\geq1$. The characterization of all pairs of
probability measures as well as the pairs of quasi-definite linear functionals
which are coherent pairs of second kind has been given in \cite{Paco}. Some
illustrative examples were shown therein.\newline

The concept of $\Delta$-coherent pair of linear functionals was introduced in
\cite{MR1785535} as well as in \cite{MR1949214}. Indeed, a pair of linear
functionals $\{L_{0},L_{1}\}$ it is said to be a $\Delta$-coherent pair if the
corresponding sequences of MOPs $\{P_{n}^{(0)}(x)\}_{n\geq0}$ and
$\{P_{n}^{(1)}(x)\}_{n\geq0}$ satisfy a discrete version of (\ref{Eq-CP})%

\begin{equation}
P_{n}^{(1)}(x)=\frac{1}{n+1}\left[  \Delta P_{n+1}^{(0)}(x)-\rho_{n}\Delta
P_{n}^{(0)}(x)\right]  ,\quad\rho_{n}\neq0,\quad n\geq1. \label{Eq-DCPLF}%
\end{equation}
It was shown in this case that the sequence of monic polynomials $S_{n}(x)$
orthogonal with respect to the inner product%

\[
\langle f,g\rangle_{\mathfrak{S}}=\langle L_{0},fg\rangle+\langle L_{1},\Delta
f\Delta g\rangle
\]
satisfies the connection formulas
\begin{equation}%
\begin{array}
[c]{l}%
S_{n+1}(x)-\gamma_{n}S_{n}(x)=P_{n+1}^{(0)}(x)-\rho_{n}P_{n}^{(0)}%
(x),\\[1.5ex]%
\Delta S_{n+1}(x)-\gamma_{n}\Delta S_{n}(x)=(n+1)P_{n}^{(1)}(x),
\end{array}
\quad n\geq1. \label{Eq-DSobolevOrth-to-NormalOrth}%
\end{equation}
In \cite{MR1949214}, it was proved that one of them must be a $\Delta
$-classical linear functional. The corresponding companion linear functionals
were described therein. Notice that this inner product is a discretization of
a Sobolev inner product (\ref{Eq-Sobolev}), i.e., $\nu_{0}$ and $\nu_{1}$ are
discrete measures.\newline

The paper is organized as follows. Section 2 summarizes the basic concepts
about linear functionals and orthogonal polynomials to be used in the sequel.
We emphasize a special family of linear functionals, the so called $\Delta
$\emph{-semiclassical}, which will play a central role along the manuscript.
In Section 3 we introduce the concept of $\Delta$-coherent pair of the second
kind for linear functionals and we give a characterization of them. We must
point out that both functionals in the pair turn out to be discrete
semiclassical functionals of class at most $1$. In Section 4 we classify all
the $\Delta$- coherent pairs of the second kind. We show that the companions
of the discrete classical functionals (Charlier, Kravchuk, Meixner, and Hahn)
are discrete semiclassical of class $s=1$.

\section{Preliminary material}

Let $\mathbb{P}=\mathbb{C}\left[  x\right]  $ and $\mathbb{N}_{0}$ be the set
of nonnegative integers%
\[
\mathbb{N}_{0}=\mathbb{N}\cup\left\{  0\right\}  =\left\{  0,1,2,\ldots
\right\}  .
\]
We will denote by $\delta_{k,n}$ the \emph{Kronecker delta}, defined by%
\[
\delta_{k,n}=\left\{
\begin{array}
[c]{c}%
1,\quad k=n\\
0,\quad k\neq n
\end{array}
\right.  ,\quad k,n\in\mathbb{N}_{0},
\]
and say that $\left\{  \Lambda_{n}\left(  x\right)  \right\}  _{n\geq0}%
\subset\mathbb{P}$ is a \emph{basis }of $\mathbb{P}$ if $\deg\left(
\Lambda_{n}\right)  =n.$

Suppose that $L:\mathbb{P}\rightarrow\mathbb{C}$ is a \emph{linear functional,
}$\left\{  \Lambda_{n}\left(  x\right)  \right\}  _{n\geq0}$ is a basis of
$\mathbb{P}$, and we choose a \textbf{nonzero} sequence $\left\{
h_{n}\right\}  _{n\geq0}\subset\mathbb{C}$. If the system of linear equations%
\begin{equation}%
{\displaystyle\sum\limits_{i=0}^{n}}
L\left[  \Lambda_{k}\Lambda_{i}\right]  c_{n,i}=h_{n}\delta_{k,n},\quad0\leq
k\leq n, \label{system}%
\end{equation}
has a \textbf{unique solution} $\left\{  c_{n,i}\right\}  _{0\leq i\leq n},$
then we can define a polynomial $P_{n}\left(  x\right)  $ by%

\[
P_{n}\left(  x\right)  =%
{\displaystyle\sum\limits_{i=0}^{n}}
c_{n,i}\Lambda_{i}\left(  x\right)  ,\quad n\in\mathbb{N}_{0}.
\]
We say that $\left\{  P_{n}\left(  x\right)  \right\}  _{n\geq0}$ is an
\emph{orthogonal polynomial sequence} with respect to the functional $L$. The
system (\ref{system}) can be written as
\[
L\left[  \Lambda_{k}P_{n}\right]  =h_{n}\delta_{k,n},\quad0\leq k\leq n,
\]
and using linearity we see that the sequence $\left\{  P_{n}\left(  x\right)
\right\}  _{n\geq0}$ satisfies the \emph{orthogonality conditions}%
\begin{equation}
L\left[  P_{k}P_{n}\right]  =h_{n}\delta_{k,n},\quad k,n\in\mathbb{N}_{0}.
\label{ortho}%
\end{equation}
The linear functional $L$ is said to be \emph{quasi-definite} \cite{MR0481884}.

We define the multiplication of a linear functional by a polynomial $q(x)$ to
be the linear functional $qL$ such that%
\begin{equation}
\left(  qL\right)  \left[  p\right]  =L\left[  qp\right]  ,\quad
p\in\mathbb{P}, \label{mult}%
\end{equation}
and the adjoint operator $\Delta^{\ast}$ of a linear functional by%
\begin{equation}
\left(  \Delta^{\ast}L\right)  \left[  p\right]  =-L\left[  \Delta p\right]
,\quad p\in\mathbb{P}. \label{adj}%
\end{equation}
We say that $L$ is a \emph{semiclassical functional} with respect to the
operator $\Delta,$ ($\Delta-$semiclassical, in short) \cite{MR1665164} if
there exist polynomials $\left(  \phi,\psi\right)  $ such that $L$ satisfies
the \emph{Pearson equation}%
\begin{equation}
\Delta^{\ast}\left(  \phi L\right)  +\psi L=0, \label{Pearson}%
\end{equation}
where $\phi$ is monic and $\deg\left(  \psi\right)  \geq1.$ Note that using
(\ref{mult}) and (\ref{adj}) we can rewrite the Pearson equation as
\begin{equation}
L\left[  \phi\Delta p\right]  =L\left[  \psi p\right]  ,\quad p\in
\mathbb{P}.\quad\label{Pearson1}%
\end{equation}

\begin{lemma}
If $f,g:\mathbb{Z}\rightarrow\mathbb{C},$ then we have the \emph{summation by
parts} formula%
\begin{equation}%
{\displaystyle\sum\limits_{x=a}^{b}}
f\left(  x\right)  \Delta g\left(  x\right)  =\left[  f\left(  x-1\right)
g\left(  x\right)  \right]  _{a}^{b+1}-%
{\displaystyle\sum\limits_{x=a}^{b}}
g\left(  x\right)  \nabla f\left(  x\right)  ,\quad a,b\in\mathbb{Z}.
\label{SBP}%
\end{equation}

\end{lemma}

\begin{proof}
The formula follows from the telescoping sum%
\begin{align*}%
{\displaystyle\sum\limits_{x=a}^{b}}
\left[  f\left(  x\right)  \Delta g\left(  x\right)  +g\left(  x\right)
\nabla f\left(  x\right)  \right]   &  =%
{\displaystyle\sum\limits_{x=a}^{b}}
f\left(  x\right)  g\left(  x+1\right)  -g\left(  x\right)  f\left(
x-1\right) \\
&  =f\left(  b\right)  g\left(  b+1\right)  -g\left(  a\right)  f\left(
a-1\right)  .
\end{align*}

\end{proof}

Using summation by parts, we can write an alternative form of the Pearson
equation (\ref{Pearson}).

\begin{proposition}
\label{Prop1}Let $L\in P^{\ast}$ be defined by%
\[
L\left[  p\right]  =%
{\displaystyle\sum\limits_{x=0}^{\infty}}
p\left(  x\right)  \rho\left(  x\right)  ,\quad p\in\mathbb{P},
\]
where $\rho\left(  -1\right)  =0$ and%
\[
\underset{x\rightarrow\infty}{\lim}p\left(  x\right)  \rho\left(  x\right)
=0,\quad p\in\mathbb{P}.
\]
Then, $L$ satisfies the Pearson equation (\ref{Pearson}) if and only if%
\begin{equation}
\nabla\left(  \phi\rho\right)  +\psi\rho=0. \label{Pearson rho}%
\end{equation}

\end{proposition}

\begin{proof}
Since (\ref{Pearson}) is equivalent to (\ref{Pearson1}), we can use
(\ref{SBP}) and conclude that $L$ satisfies (\ref{Pearson}) if and only if%
\[%
{\displaystyle\sum\limits_{x=0}^{\infty}}
\psi\left(  x\right)  p\left(  x\right)  \rho\left(  x\right)  =-%
{\displaystyle\sum\limits_{x=0}^{\infty}}
p\left(  x\right)  \nabla\left(  \phi\rho\right)  \left(  x\right)  ,\quad
p\in\mathbb{P}.
\]
If the last equation holds for all $p\in\mathbb{P},$ then we must have%
\[
\psi\rho=-\nabla\left(  \phi\rho\right)  ,
\]
and the result follows.
\end{proof}

The \emph{class} of a discrete semiclassical linear functional $L$ is defined
by%
\begin{equation}
s=\min\{\max\left\{  \deg\left(  \phi\right)  -2,\ \deg\left(  \psi\right)
\}-1\right\}  , \label{class}%
\end{equation}
where the minimum is taken among all pairs of polynomials $\phi,\psi$ such
that the Pearson equation (\ref{Pearson}) holds for the linear functional $L.$
Functionals of class $s=0$ are called $\Delta$-classical (see \cite{MR1340932}%
, \cite{MR2656096}, and \cite{MR1149380} among others). $\Delta$-semiclassical
linear functionals of class $s=1$ have been described in \cite{MR3227440} and
\cite{MR4238531}.

\begin{remark}
Let $\deg\left(  \phi\right)  =d_{1},$ $\deg\left(  \psi\right)  =d_{2}.$ If
$d_{1}=d_{2}+1=s+2,$ we will always assume the admissibility condition
\begin{equation}
\psi_{d_{2}}\neq n-s,\quad n\in\mathbb{N}_{0}, \label{admiss}%
\end{equation}
where $\psi_{d_{2}}$ is the leading coefficient of $\psi\left(  x\right)  .$
\end{remark}

We will denote by $\{\varphi_{n}\left(  x\right)  \}_{n\geq0} $ the basis of
\emph{falling factorial polynomials} defined by $\varphi_{0}\left(  x\right)
=1$ and%
\begin{equation}
\varphi_{n}\left(  x\right)  =%
{\displaystyle\prod\limits_{k=0}^{n-1}}
\left(  x-k\right)  ,\quad n\in\mathbb{N}. \label{phi def}%
\end{equation}
The polynomials $\varphi_{n}\left(  x\right)  $ satisfy the basic identities%
\begin{equation}
x\varphi_{n}\left(  x\right)  =\varphi_{n+1}\left(  x\right)  +n\varphi
_{n}\left(  x\right)  ,\quad n\in\mathbb{N}_{0}, \label{req}%
\end{equation}
and
\begin{equation}
\Delta\varphi_{n}=n\varphi_{n-1},\quad n\in\mathbb{N}_{0}. \label{diff phi}%
\end{equation}

\begin{lemma}
We have the representation%
\begin{equation}
\varphi_{n}\left(  x-1\right)  =%
{\displaystyle\sum\limits_{k=0}^{n}}
\left(  -1\right)  ^{k}\varphi_{k}\left(  n\right)  \varphi_{n-k}\left(
x\right)  ,\quad n\in\mathbb{N}_{0}. \label{phi x-1}%
\end{equation}

\end{lemma}

\begin{proof}
We use induction. The identity is true for $n=0$ since $\varphi_{0}\left(
x\right)  =1.$ Assuming (\ref{phi x-1}) to be true for all $0\leq k\leq n$ and
using the recurrence (\ref{req}), we have%
\begin{align*}
\varphi_{n+1}\left(  x-1\right)   &  =\left(  x-1-n\right)  \varphi_{n}\left(
x-1\right)  =%
{\displaystyle\sum\limits_{k=0}^{n}}
\left(  -1\right)  ^{k}\varphi_{k}\left(  n\right)  \left(  x-1-n\right)
\varphi_{n-k}\left(  x\right) \\
&  =%
{\displaystyle\sum\limits_{k=0}^{n}}
\left(  -1\right)  ^{k}\varphi_{k}\left(  n\right)  \left[  \varphi
_{n+1-k}\left(  x\right)  -\left(  k+1\right)  \varphi_{n-k}\left(  x\right)
\right]  .
\end{align*}
Since $\varphi_{n+1}\left(  n\right)  =0,$ we can rewrite the sum above as%
\[
\varphi_{n+1}\left(  x-1\right)  =%
{\displaystyle\sum\limits_{k=0}^{n+1}}
\left(  -1\right)  ^{k}\varphi_{k}\left(  n\right)  \varphi_{n+1-k}\left(
x\right)  +%
{\displaystyle\sum\limits_{k=0}^{n}}
\left(  -1\right)  ^{k+1}\left(  k+1\right)  \varphi_{k}\left(  n\right)
\varphi_{n-k}\left(  x\right)  ,
\]
or shifting the index in the second sum%
\[
\varphi_{n+1}\left(  x-1\right)  =%
{\displaystyle\sum\limits_{k=0}^{n+1}}
\left(  -1\right)  ^{k}\varphi_{k}\left(  n\right)  \varphi_{n+1-k}\left(
x\right)  +%
{\displaystyle\sum\limits_{k=1}^{n+1}}
\left(  -1\right)  ^{k}k\varphi_{k-1}\left(  n\right)  \varphi_{n+1-k}\left(
x\right)  .
\]

But from (\ref{diff phi}) we see that%
\[
\varphi_{k}\left(  n+1\right)  =\varphi_{k}\left(  n\right)  +k\varphi
_{k-1}\left(  n\right)  ,
\]
and we conclude that
\[
\varphi_{n+1}\left(  x-1\right)  =%
{\displaystyle\sum\limits_{k=0}^{n+1}}
\left(  -1\right)  ^{k}\varphi_{k}\left(  n+1\right)  \varphi_{n+1-k}\left(
x\right)  .
\]

\end{proof}

\begin{proposition}
\label{Prop2}Let $\{p_{n}\left(  x\right)  \}_{n\geq0}$ be a sequence of
polynomials orthogonal with respect to a $\Delta-$semiclassical functional $L$
of class $s.$ Then,%
\begin{equation}
L\left[  \phi\Delta q\Delta p_{n}\right]  =0,\quad q\in\mathbb{P},\quad
\deg\left(  q\right)  <n-s. \label{ortho1}%
\end{equation}

\end{proposition}

\begin{proof}
The equation (\ref{ortho1}) is obviously true if $q=1,$ so we can assume that
$\deg\left(  q\right)  \geq1.$ Using (\ref{phi x-1}) with $k\geq1,$ we have%
\[
\Delta\varphi_{k}=-%
{\displaystyle\sum\limits_{j=1}^{k}}
\left(  -1\right)  ^{j}\varphi_{j}\left(  k\right)  \varphi_{k-j}\left(
x+1\right)  ,
\]
and multiplying by $\Delta p_{n}$ we get
\[
\Delta\varphi_{k}\Delta p_{n}=%
{\displaystyle\sum\limits_{j=1}^{k}}
\left(  -1\right)  ^{j}\varphi_{j}\left(  k\right)  \left[  \Delta
\varphi_{k-j}p_{n}-\Delta\left(  \varphi_{k-j}p_{n}\right)  \right]  ,
\]
where we have used the identity%
\begin{equation}
\Delta\left(  fg\right)  =g\left(  x+1\right)  \Delta f+f\Delta g=g\Delta
f+f\Delta g+\Delta f\Delta g. \label{diff prod}%
\end{equation}

Using the Pearson equation (\ref{Pearson}) and (\ref{diff phi}), we obtain
\begin{align*}
L\left[  \phi\Delta\varphi_{k}\Delta p_{n}\right]   &  =%
{\displaystyle\sum\limits_{j=1}^{k}}
\left(  -1\right)  ^{j}\varphi_{j}\left(  k\right)  \left(  L\left[
\phi\Delta\varphi_{k-j}p_{n}\right]  -L\left[  \psi\varphi_{k-j}p_{n}\right]
\right) \\
&  =%
{\displaystyle\sum\limits_{j=1}^{k}}
\left(  -1\right)  ^{j}\varphi_{j}\left(  k\right)  \left\{  \left(
k-j\right)  L\left[  \phi\varphi_{k-j-1}p_{n}\right]  -L\left[  \psi
\varphi_{k-j}p_{n}\right]  \right\}  ,
\end{align*}
and since the sequence of polynomials $\{p_{n}(x)\}_{n\geq0}$ is orthogonal
with respect to $L$ we conclude that%
\[
L\left[  \phi\Delta\varphi_{k}\Delta p_{n}\right]  =0,\quad k<\min\left\{
n+2-\deg\left(  \phi\right)  ,\ n+1-\deg\left(  \psi\right)  \right\}  =n-s.
\]
Using the fact that $\left\{  \varphi_{k}\left(  x\right)  \right\}  _{k\geq
0}$ is a polynomial basis, the result follows.
\end{proof}

Using Proposition \ref{Prop2}, we can prove the following characterization of
discrete semiclassical polynomials.

\begin{theorem}
Let $\{p_{n}\left(  x\right)  \}_{n\geq0}$ be a sequence of monic polynomials
orthogonal with respect to a functional $L.$ The following statements are equivalent:
\end{theorem}

\begin{proposition}
(i) $L$ is $\Delta-$semiclassical of class $s.$

(ii) There exists a monic polynomial $\phi\left(  x\right)  $ such that the
polynomials $p_{n}$ satisfy the \emph{structure equation}%
\begin{equation}
\phi\left(  x\right)  \Delta p_{n+1}=%
{\displaystyle\sum\limits_{k=n-s}^{n+d_{1}}}
\epsilon_{n,k}p_{k},\quad n>s,\quad\epsilon_{n,n-s}\neq0. \label{structure}%
\end{equation}

\end{proposition}

\begin{proof}
(i) $\Rightarrow$ (ii) Since the sequence $\{p_{n}\}_{n\geq0}$ is a polynomial
basis in the linear space of polynomials, we can write
\begin{equation}
\phi\left(  x\right)  \Delta p_{n+1}=%
{\displaystyle\sum\limits_{k=0}^{n+d_{1}}}
\epsilon_{n,k}p_{k}. \label{sum}%
\end{equation}
Using the orthogonality relation (\ref{ortho}) and (\ref{diff prod}), we have%
\[
h_{k}\epsilon_{n,k}=L\left[  \phi p_{k}\Delta p_{n+1}\right]  =L\left[
\phi\Delta\left(  p_{k}p_{n+1}\right)  \right]  -L\left[  \phi\Delta
p_{k}p_{n+1}\right]  -L\left[  \phi\Delta p_{k}\Delta p_{n+1}\right]  .
\]

Using (\ref{Pearson1}) and orthogonality, we get%
\[
L\left[  \phi\Delta\left(  p_{k}p_{n+1}\right)  \right]  =L\left[  \psi
p_{k}p_{n+1}\right]  =0,\quad k<n+1-d_{2},
\]
and%
\[
L\left[  \phi\Delta\varphi_{k}p_{n+1}\right]  =kL\left[  \phi\varphi
_{k-1}p_{n+1}\right]  =0,\quad k<n+2-d_{1},
\]
from which it follows that
\[
L\left[  \phi\Delta\left(  p_{k}p_{n+1}\right)  \right]  -L\left[  \phi\Delta
p_{k}p_{n+1}\right]  =0,\quad k<n-s.
\]
Using (\ref{ortho1}) we see that%
\[
L\left[  \phi\Delta p_{k}\Delta p_{n+1}\right]  =0,\quad k<n+1-s,
\]
and therefore%
\[
\epsilon_{n,k}=0,\quad k<n-s.
\]
For $k=n-s,$ we have
\[
\epsilon_{n,n-s}=h_{n+1}\left(  \frac{\psi_{d_{2}}}{h_{n+1-d_{2}}}%
\delta_{d_{2},s+1}-\frac{n+2-d_{1}}{h_{n+2-d_{1}}}\delta_{d_{1},s+2}\right)
,
\]
and it follows that $\epsilon_{n,n-s}\neq0$ as long as $d_{1}-2\neq d_{2}-1.$
If $d_{1}=d_{2}+1,$ we need the additional condition (\ref{admiss}).

(ii) $\Rightarrow$ (i) Since $\left\{  \frac{p_{n}}{h_{n}}L\right\}  _{n\geq
0}$ is a basis of $\mathbb{P}^{\ast}$ satisfying%
\[
\frac{p_{n}}{h_{n}}L\left[  p_{k}\right]  =\frac{1}{h_{n}}L\left[  p_{k}%
p_{n}\right]  =\delta_{k,n},
\]
we have the representation%
\[
\Delta^{\ast}\left(  \phi L\right)  =%
{\displaystyle\sum\limits_{n=0}^{\infty}}
c_{n}\frac{p_{n}}{h_{n}}L.
\]
Using the structure relation (\ref{structure}) and orthogonality, we get%
\begin{align*}
c_{n+1}  &  =\Delta^{\ast}\left(  \phi L\right)  \left[  p_{n+1}\right]
=-L\left[  \phi\Delta p_{n+1}\right] \\
&  =-%
{\displaystyle\sum\limits_{k=n-s}^{n+d_{1}}}
\epsilon_{n,k}L\left[  p_{k}\right]  =\left\{
\begin{array}
[c]{c}%
0,\quad n>s\\
-\mu_{0}\epsilon_{n,0},\quad n\leq s
\end{array}
\right.  ,
\end{align*}
where $\mu_{0}=L\left[  1\right]  .$ Thus, $\Delta^{\ast}\left(  \phi
L\right)  +\psi L=0,$ with%
\[
\psi\left(  x\right)  =\mu_{0}%
{\displaystyle\sum\limits_{n=1}^{s+1}}
\frac{\epsilon_{n-1,0}}{h_{n}}p_{n}\left(  x\right)  .
\]
Note that $\deg\left(  \psi\right)  -1\leq s,$ in agreement with (\ref{class}).
\end{proof}

For more information about this kind of structure relations for $\Delta
$-semiclassical linear functionals, check \cite{MR1665164}. Notice that in the
$\Delta$-classical case $\left(  s=0\right)  $ two structure relations
characterizing such linear functionals are given in \cite{MR1340932}.

\section{Coherent pairs}

Let $\left\{  L_{0},L_{1}\right\}  $ be a pair of quasi-definite functionals
and $\left\{  P_{n}^{\left(  0\right)  }\left(  x\right)  \right\}  _{n\geq
0},\left\{  P_{n}^{\left(  1\right)  }\left(  x\right)  \right\}  _{n\geq0}$
be the corresponding sequences of monic orthogonal polynomials. We say that
$\left\{  L_{0},L_{1}\right\}  $ is a $\Delta-$\emph{coherent pair of the
second kind} (abbreviated $\Delta\mathrm{c}2)$ if there exists $\left\{
\tau_{n}\right\}  _{n\geq0}\subset\mathbb{C}\setminus\left\{  0\right\}  $
such that
\begin{equation}
Q_{n}(x)=\frac{\Delta P_{n+1}^{\left(  0\right)  }(x)}{n+1}=P_{n}^{\left(
1\right)  }(x)-\tau_{n}P_{n-1}^{\left(  1\right)  } (x),\quad n\in
\mathbb{N}_{0}, \label{coherent}%
\end{equation}
where $P_{-1}^{\left(  0\right)  }(x)=P_{-1}^{\left(  1\right)  }(x)=0.$ The
sequences of polynomials $\left\{  P_{n}^{\left(  i\right)  }\left(  x\right)
\right\}  _{n\geqq0},$ $i=0,1,$ satisfy the orthogonality conditions%
\begin{equation}
L_{i}\left[  P_{k}^{\left(  i\right)  }P_{n}^{\left(  i\right)  }\right]
=h_{n}^{\left(  i\right)  }\delta_{n,k},\quad i=0,1,\quad n,k\in\mathbb{N}%
_{0}. \label{ortho2}%
\end{equation}

\begin{proposition}
\label{Prop3}If $\left\{  L_{0},L_{1}\right\}  $ is a $\Delta\mathrm{c}2$ and
the functionals $\mathbf{v}_{n}^{\left(  i\right)  }$ are defined by
\begin{equation}
\mathbf{v}_{n}^{\left(  i\right)  }=\frac{P_{n}^{\left(  i\right)  }}%
{h_{n}^{\left(  i\right)  }}L_{i},\quad i=0,1,\quad n\in\mathbb{N}_{0},
\label{vn def}%
\end{equation}
then
\begin{equation}
\Delta^{\ast}\mathbf{v}_{n}^{\left(  1\right)  }=\tau_{n+1}\left(  n+2\right)
\mathbf{v}_{n+2}^{\left(  0\right)  }-\left(  n+1\right)  \mathbf{v}%
_{n+1}^{\left(  0\right)  },\quad n\in\mathbb{N}_{0}. \label{D v1}%
\end{equation}

\end{proposition}

\begin{proof}
Using (\ref{ortho2}) and (\ref{vn def}), we have%
\begin{equation}
\mathbf{v}_{n}^{\left(  i\right)  }\left[  P_{k}^{\left(  i\right)  }\right]
=\delta_{n,k},\quad i=0,1,\quad n,k\in\mathbb{N}_{0}. \label{dual def}%
\end{equation}
Let $\mathbf{u}_{n}\in\mathbb{P}^{\ast}$ be defined by%
\begin{equation}
\mathbf{u}_{n}\left[  Q_{k}\right]  =\delta_{n,k},\quad n,k\in\mathbb{N}_{0}.
\label{dual un}%
\end{equation}
Since the sequence of linear functionals $\{\mathbf{u}_{n}\}_{n\geq0}$ is a
basis of $\mathbb{P}^{\ast},$ we can write%
\[
\mathbf{v}_{n}^{\left(  1\right)  }=%
{\displaystyle\sum\limits_{k=0}^{\infty}}
a_{n,k}\mathbf{u}_{n}.
\]
Using (\ref{coherent}) and (\ref{dual def}), we get%
\[
a_{n,k}=\mathbf{v}_{n}^{\left(  1\right)  }\left[  Q_{k}\right]
=\mathbf{v}_{n}^{\left(  1\right)  }\left[  P_{k}^{\left(  1\right)  }%
-\tau_{k}P_{k-1}^{\left(  1\right)  }\right]  =\delta_{n,k}-\tau_{k}%
\delta_{n,k-1},
\]
and therefore%
\begin{equation}
\mathbf{v}_{n}^{\left(  1\right)  }=\mathbf{u}_{n}-\tau_{n+1}\mathbf{u}%
_{n+1},\quad n\in\mathbb{N}_{0}. \label{v1}%
\end{equation}

If we write%
\[
\Delta^{\ast}\mathbf{u}_{n}=%
{\displaystyle\sum\limits_{k=0}^{\infty}}
b_{n,k}\mathbf{v}_{k}^{\left(  0\right)  },
\]
then%
\[
b_{n,k}=\Delta^{\ast}\mathbf{u}_{n}\left[  P_{k}^{\left(  0\right)  }\right]
=-\mathbf{v}_{n}^{\left(  0\right)  }\left[  \Delta P_{k}^{\left(  0\right)
}\right]  =-k\mathbf{v}_{n}^{\left(  0\right)  }\left[  Q_{k-1}\right]
=-k\delta_{n,k-1},
\]
and, as a consequence,
\begin{equation}
\Delta^{\ast}\mathbf{u}_{n}=-\left(  n+1\right)  \mathbf{v}_{n+1}^{\left(
0\right)  }. \label{Dvo}%
\end{equation}
Using (\ref{Dvo}) in (\ref{v1}), we obtain (\ref{D v1}).
\end{proof}

With the help of Proposition \ref{Prop3}, we can prove a characterization of
$\Delta-$coherent pairs of the second kind.

\begin{theorem}
\label{Th1}The following are equivalent:

(1) $\left\{  L_{0},L_{1}\right\}  $ is a $\Delta\mathrm{c}2$.

(2) There exist $\Lambda_{2},\Lambda_{3}\in\mathbb{P}$ defined by
\begin{equation}
\Lambda_{k}\left(  x\right)  =%
{\displaystyle\sum\limits_{i=0}^{k}}
\lambda_{i}^{\left(  k\right)  }x^{i},\quad k=2,3, \label{Lambda def}%
\end{equation}
and satisfying%
\begin{equation}
\lambda_{2}^{\left(  2\right)  }+\left(  n-1\right)  \lambda_{3}^{\left(
3\right)  }\neq0,\quad n\in\mathbb{N}, \label{admis}%
\end{equation}
such that
\begin{equation}
\Delta^{\ast}L_{1}=\Lambda_{2}L_{0},\quad L_{1}=\Lambda_{3}L_{0}.
\label{L0, L1}%
\end{equation}

\end{theorem}

\begin{proof}
(1) $\Rightarrow$ (2) Setting $n=0$ in (\ref{D v1}), we have%
\[
\frac{1}{h_{0}^{\left(  1\right)  }}\Delta^{\ast}L_{1}=2\tau_{1}\frac
{P_{2}^{\left(  0\right)  }}{h_{2}^{\left(  0\right)  }}L_{0}-\frac
{P_{1}^{\left(  0\right)  }}{h_{1}^{\left(  0\right)  }}L_{0},
\]
and defining
\[
\Lambda_{2}\left(  x\right)  =\frac{2\tau_{1}h_{0}^{\left(  1\right)  }}%
{h_{2}^{\left(  0\right)  }}P_{2}^{\left(  0\right)  }\left(  x\right)
-\frac{h_{0}^{\left(  1\right)  }}{h_{1}^{\left(  0\right)  }}P_{1}^{\left(
0\right)  }\left(  x\right)  ,
\]
we get $\Delta^{\ast}L_{1}=\Lambda_{2}L_{0}.$ Since $\tau_{1}\neq0,$ we see
that $\deg\left(  \Lambda_{2}\right)  =2.$

Similarly, setting $n=1$ in (\ref{D v1}) gives%
\[
\frac{1}{h_{1}^{\left(  1\right)  }}\Delta^{\ast}\left(  P_{1}^{\left(
1\right)  }L_{1}\right)  =3\tau_{2}\frac{P_{3}^{\left(  0\right)  }}%
{h_{3}^{\left(  0\right)  }}L_{0}-2\frac{P_{2}^{\left(  0\right)  }}%
{h_{2}^{\left(  0\right)  }}L_{0}%
\]
and using (\ref{diff prod}) we have%
\[
\Delta^{\ast}P_{1}^{\left(  1\right)  }L_{1}+P_{1}^{\left(  1\right)  }%
\Delta^{\ast}L_{1}+\Delta^{\ast}P_{1}^{\left(  1\right)  }\Delta^{\ast}%
L_{1}=\left(  \frac{3\tau_{2}h_{1}^{\left(  1\right)  }}{h_{3}^{\left(
0\right)  }}P_{3}^{\left(  0\right)  }-\frac{2h_{1}^{\left(  1\right)  }%
}{h_{2}^{\left(  0\right)  }}P_{2}^{\left(  0\right)  }\right)  L_{0}.
\]
Since $\Delta^{\ast}P_{1}^{\left(  1\right)  }=1$ and $\Delta^{\ast}%
L_{1}=\Lambda_{2}L_{0},$ we conclude that $L_{1}=\Lambda_{3}L_{0}$ with%
\[
\Lambda_{3}\left(  x\right)  =\frac{3\tau_{2}h_{1}^{\left(  1\right)  }}%
{h_{3}^{\left(  0\right)  }}P_{3}^{\left(  0\right)  }\left(  x\right)
-\frac{2h_{1}^{\left(  1\right)  }}{h_{2}^{\left(  0\right)  }}P_{2}^{\left(
0\right)  }\left(  x\right)  -\Lambda_{2}\left(  x\right)  \left[
P_{1}^{\left(  1\right)  }\left(  x\right)  +1\right]  .
\]
Note that $\deg\left(  \Lambda_{3}\right)  \leq3.$

(2) $\Rightarrow$ (1) Since the polynomial sequence $\{P_{n}^{\left(
1\right)  }\left(  x\right)  \}_{n\geq0} $ is a basis of $\mathbb{P}$, we
have
\[
Q_{n}(x)=P_{n}^{\left(  1\right)  }+%
{\displaystyle\sum\limits_{k=0}^{n-1}}
c_{n,k}P_{k}^{\left(  1\right)  } (x).
\]
Using orthogonality and (\ref{diff prod}), we get%
\begin{align*}
&  \left(  n+1\right)  h_{k}^{\left(  1\right)  }c_{n,k}=\left(  n+1\right)
L_{1}\left[  Q_{n}P_{k}^{\left(  1\right)  }\right]  =L_{1}\left[  \Delta
P_{n+1}^{\left(  0\right)  }P_{k}^{\left(  1\right)  }\right] \\
&  =L_{1}\left[  \Delta\left(  P_{n+1}^{\left(  0\right)  }P_{k}^{\left(
1\right)  }\right)  \right]  -L_{1}\left[  P_{n+1}^{\left(  0\right)  }\Delta
P_{k}^{\left(  1\right)  }\right]  -L_{1}\left[  \Delta P_{n+1}^{\left(
0\right)  }\Delta P_{k}^{\left(  1\right)  }\right] \\
&  =-L_{0}\left[  \Lambda_{2}P_{n+1}^{\left(  0\right)  }P_{k}^{\left(
1\right)  }\right]  -L_{0}\left[  \Lambda_{3}P_{n+1}^{\left(  0\right)
}\Delta P_{k}^{\left(  1\right)  }\right]  -L_{0}\left[  \Lambda_{3}\Delta
P_{n+1}^{\left(  0\right)  }\Delta P_{k}^{\left(  1\right)  }\right]  ,
\end{align*}
and hence for all $0\leq k\leq n-1$
\begin{align*}
L_{0}\left[  \Lambda_{2}P_{n+1}^{\left(  0\right)  }P_{k}^{\left(  1\right)
}\right]   &  =\lambda_{2}^{\left(  2\right)  }h_{n+1}^{\left(  0\right)
}\delta_{k,n-1},\\
L_{0}\left[  \Lambda_{3}P_{n+1}^{\left(  0\right)  }\Delta P_{k}^{\left(
1\right)  }\right]   &  =\left(  n-1\right)  \lambda_{3}^{\left(  3\right)
}h_{n+1}^{\left(  0\right)  }\delta_{k,n-1}.
\end{align*}

From (\ref{L0, L1}) it follows that
\[
\Delta^{\ast}\left(  \Lambda_{3}L_{0}\right)  =\Lambda_{2}L_{0}.
\]
Therefore $L_{0}$ is $\Delta-$semiclassical of class at most $1.$ Using
(\ref{ortho1}), we conclude that%
\[
L_{0}\left[  \Lambda_{3}\Delta P_{n+1}^{\left(  0\right)  }\Delta
P_{k}^{\left(  1\right)  }\right]  =0,\quad k<n.
\]
Thus, for all $0\leq k\leq n-1$%
\[
\left(  n+1\right)  h_{k}^{\left(  1\right)  }c_{n,k}=-\left[  \lambda
_{2}^{\left(  2\right)  }+\left(  n-1\right)  \lambda_{3}^{\left(  3\right)
}\right]  h_{n+1}^{\left(  0\right)  }\delta_{k,n-1},
\]
and we obtain%
\[
Q_{n} (x)=P_{n}^{\left(  1\right)  } (x)-\tau_{n}P_{n-1}^{\left(  1\right)
}(x),
\]
where%
\begin{equation}
\tau_{n}=\frac{\lambda_{2}^{\left(  2\right)  }+\left(  n-1\right)
\lambda_{3}^{\left(  3\right)  }}{n+1}\frac{h_{n+1}^{\left(  0\right)  }%
}{h_{n-1}^{\left(  1\right)  }},\quad n\in\mathbb{N}. \label{tau}%
\end{equation}
Since the polynomials $\Lambda_{2},\Lambda_{3}$ satisfy (\ref{admis}), we see
that $\tau_{n}\neq0.$
\end{proof}

\section{Classification}

In this section, we will find all $\Delta-$coherent pairs of the second kind.
In view of Theorem \ref{Th1}, it is enough to consider discrete semiclassical
functionals of class $s\leq1$ satisfying the Pearson equation%
\begin{equation}
\Delta^{\ast}\left(  \Lambda_{3}L_{0}\right)  =\Lambda_{2}L_{0},\quad
\deg\left(  \Lambda_{2}\right)  =2,\quad\deg\left(  \Lambda_{3}\right)  \leq3.
\label{Pearson2}%
\end{equation}

In \cite{MR3227440}, we classified the discrete semiclassical functionals of
class $s\leq1$ satisfying (\ref{Pearson rho}), and obtained the following cases:

\begin{enumerate}
\item Charlier (class $s=0)$%
\begin{equation}
\rho\left(  x\right)  =\frac{z^{x}}{x!}, \label{Charlier}%
\end{equation}%
\begin{equation}
\phi\left(  x\right)  =1,\quad\psi\left(  x\right)  =\frac{x}{z}-1.
\label{C pearson}%
\end{equation}

\item Meixner (class $s=0)$%
\begin{equation}
\rho\left(  x\right)  =\left(  a\right)  _{x}\frac{z^{x}}{x!}, \label{Meixner}%
\end{equation}%
\begin{equation}
\phi\left(  x\right)  =x+a,\quad\psi\left(  x\right)  =\frac{x}{z}-\left(
x+a\right)  . \label{M Pearson}%
\end{equation}

Subcase: Kravchuk polynomials
\begin{equation}
\rho\left(  x\right)  =\left(  -N\right)  _{x}\frac{z^{x}}{x!},\quad
N\in\mathbb{N}, \label{Kravchouk}%
\end{equation}%
\begin{equation}
\phi\left(  x\right)  =x-N,\quad\psi\left(  x\right)  =\frac{x}{z}-\left(
x-N\right)  . \label{Krav Pearson}%
\end{equation}

\item Hahn (class $s=0)$%
\begin{equation}
\rho\left(  x\right)  =\frac{\left(  -N\right)  _{x}\left(  a\right)  _{x}%
}{\left(  b+1\right)  _{x}}\frac{1}{x!},\quad N\in\mathbb{N}, \label{Hahn}%
\end{equation}%
\begin{equation}
\phi\left(  x\right)  =\left(  x-N\right)  \left(  x+a\right)  ,\quad
\psi\left(  x\right)  =\left(  N-a+b\right)  x+aN . \label{Hahn Pearson}%
\end{equation}

\item Generalized Charlier (class $s=1)$%
\begin{equation}
\rho\left(  x\right)  =\frac{1}{\left(  b+1\right)  _{x}}\frac{z^{x}}{x!},
\label{Gen Charlier}%
\end{equation}%
\begin{equation}
\phi\left(  x\right)  =1,\quad\psi\left(  x\right)  =\frac{x\left(
x+b\right)  }{z}-1. \label{GC pearson}%
\end{equation}

\item Generalized Meixner (class $s=1)$%
\begin{equation}
\rho\left(  x\right)  =\frac{\left(  a\right)  _{x}}{\left(  b+1\right)  _{x}%
}\frac{z^{x}}{x!}, \label{Gen Meixner}%
\end{equation}%
\begin{equation}
\phi\left(  x\right)  =x+a,\quad\psi\left(  x\right)  =\frac{x\left(
x+b\right)  }{z}-\left(  x+a\right)  . \label{GM Pearson}%
\end{equation}

\item Generalized Kravchuk (class $s=1)$%
\begin{equation}
\rho\left(  x\right)  =\left(  -N\right)  _{x}\left(  a\right)  _{x}%
\frac{z^{x}}{x!},\quad N\in\mathbb{N}, \label{Gen Krav}%
\end{equation}%
\begin{equation}
\phi\left(  x\right)  =\left(  x-N\right)  \left(  x+a\right)  ,\quad
\psi\left(  x\right)  =\frac{x}{z}-\left(  x-N\right)  \left(  x+a\right)  .
\label{Gen Krav Pearson}%
\end{equation}

\item Generalized Hahn of type I (class $s=1)$%
\begin{equation}
\rho\left(  x\right)  =\frac{\left(  a_{1}\right)  _{x}\left(  a_{2}\right)
_{x}}{\left(  b+1\right)  _{x}}\frac{z^{x}}{x!}, \label{G Hahn}%
\end{equation}%
\begin{equation}
\phi\left(  x\right)  =\left(  x+a_{1}\right)  \left(  x+a_{2}\right)
,\quad\psi\left(  x\right)  =\frac{x\left(  x+b\right)  }{z}-\left(
x+a_{1}\right)  \left(  x+a_{2}\right)  . \label{G Hahn Pearson}%
\end{equation}
Note that if we set $z=1,a_{1}=a,a_{2}=-N,$ then we obtain the Hahn polynomials.

\item Generalized Hahn of type II (class $s=1)$%
\begin{equation}
\rho\left(  x\right)  =\frac{\left(  -N\right)  _{x}\left(  a_{1}\right)
_{x}\left(  a_{2}\right)  _{x}}{\left(  b_{1}+1\right)  _{x}\left(
b_{2}+1\right)  _{x}}\frac{1}{x!},\quad N\in\mathbb{N} , \label{G Hanh II}%
\end{equation}%
\begin{align}
\phi\left(  x\right)   &  =\left(  x-N\right)  \left(  x+a_{1}\right)  \left(
x+a_{2}\right)  ,\label{G Hahn II Pearson}\\
\psi\left(  x\right)   &  =\left(  N-a_{1}-a_{2}+b_{1}+b_{2}\right)
x^{2}+\left(  Na_{1}+Na_{2}-a_{1}a_{2}+b_{1}b_{2}\right)  \allowbreak
x+Na_{1}a_{2}\nonumber
\end{align}

\end{enumerate}

If we define $L_{0},L_{1}\in\mathbb{P}^{\ast}$ by%
\[
L_{i}\left[  p\right]  =%
{\displaystyle\sum\limits_{x=0}^{\infty}}
p\left(  x\right)  \rho_{i}\left(  x\right)  ,\quad i=0,1,\quad p\in
\mathbb{P},
\]
with $\rho_{i}\left(  -1\right)  =0,$ $i=0,1,$ then we can use (\ref{Prop1})
and rewrite (\ref{L0, L1}) as%
\begin{equation}
\nabla\left(  \Lambda_{3}\rho_{0}\right)  +\Lambda_{2}\rho_{0}=0,\quad\rho
_{1}=\Lambda_{3}\rho_{0}. \label{r0, r1}%
\end{equation}

\begin{proposition}
\label{Prop I}Let $q,\phi_{0},\psi_{0}\in\mathbb{P}.$ If $L_{0}\in
\mathbb{P}^{\ast}$ satisfies%
\begin{equation}
\nabla\left(  \phi_{0}\rho_{0}\right)  +\psi_{0}\rho_{0}=0, \label{Pearson3}%
\end{equation}
and%
\begin{equation}
\Lambda_{2}=\left(  q-\nabla q\right)  \psi_{0}-\phi_{0}\nabla q,\quad
\Lambda_{3}=q\phi_{0}, \label{L2, L3}%
\end{equation}
then%
\begin{equation}
\nabla\left(  \Lambda_{3}\rho_{0}\right)  +\Lambda_{2}\rho_{0}=0.
\label{Pearson4}%
\end{equation}

\end{proposition}

\begin{proof}
If we use the identity%
\begin{equation}
\nabla\left(  fg\right)  =g\nabla f+f\nabla g-\nabla f\nabla g,
\label{diff prod 1}%
\end{equation}
and (\ref{L2, L3}), we have%
\[
\nabla\left(  \Lambda_{3}\rho_{0}\right)  =\left(  \nabla q\right)  \phi
_{0}\rho_{0}+\left(  q-\nabla q\right)  \nabla\left(  \phi_{0}\rho_{0}\right)
.
\]
Using (\ref{Pearson3}), we get
\[
\nabla\left(  \Lambda_{3}\rho_{0}\right)  =\left(  \nabla q\right)  \phi
_{0}\rho_{0}-\left(  q-\nabla q\right)  \psi_{0}\rho_{0}=-\Lambda_{2}\rho
_{0},
\]
and (\ref{Pearson4}) follows.
\end{proof}

We will now find all $\Delta\mathrm{c}2$ pairs with $\deg\left(  \Lambda
_{2}\right)  =2.$

\subsection{Case I: $\deg\left(  \Lambda_{3}\right)  =0$}

If we take $\Lambda_{3}=1,$ we see from (\ref{r0, r1}) that%
\[
\nabla\rho_{0}+\Lambda_{2}\rho_{0}=0,\quad\rho_{1}=\rho_{0},
\]
and if $\rho_{0}\left(  x\right)  $ is the weight function corresponding to
the generalized Charlier polynomials (\ref{Gen Charlier}), we see from
(\ref{GC pearson}) that$\quad$%
\[
\Lambda_{2}\left(  x\right)  =\frac{x\left(  x+b\right)  }{z}-1,\quad z\neq0.
\]
Thus, the generalized Charlier polynomials are \emph{self-coherent }of the
second kind\emph{. }Note that this was already observed in \cite{Almeria}.

\subsection{Case II (a): $\deg\left(  \Lambda_{3}\right)  =1$}

If we take $\Lambda_{3}=x+\omega,$ $\omega\in\mathbb{C},$ then we see from
(\ref{r0, r1}) that%
\[
\nabla\left[  \left(  x+\omega\right)  \rho_{0}\right]  +\Lambda_{2}\rho
_{0}=0,\quad\rho_{1}=\left(  x+\omega\right)  \rho.
\]
If $\rho_{0}\left(  x\right)  $ is the weight function corresponding to the
Charlier polynomials (\ref{Charlier}), we can use Proposition \ref{Prop I}
with $q=x+\omega$ and obtain%
\[
\Lambda_{2}\left(  x\right)  =\frac{x}{z}\left(  x+\omega-1\right)  -\left(
x+\omega\right)  ,\quad z\neq0,
\]
where we have also used (\ref{C pearson}). The functional $L_{1}$ is a
\emph{Christoffel transformation} of $L_{0}$ satisfying%
\[
\nabla\left[  \left(  x+\omega\right)  \rho_{1}\right]  +\Lambda_{2}\rho
_{1}=0,
\]
and using (\ref{GM Pearson}) we see that $\rho_{1}\left(  x\right)  $ is the
weight function corresponding to the generalized Meixner polynomials
(\ref{Gen Meixner}) with%
\[
a=\omega,\quad b=\omega-1.
\]

\subsection{Case II (b): $\deg\left(  \Lambda_{3}\right)  =1$}

If we take $\Lambda_{3}=x+\omega,$ $\omega\in\mathbb{C},$ then from
(\ref{r0, r1}) we get that%
\[
\nabla\left[  \left(  x+\omega\right)  \rho_{0}\right]  +\Lambda_{2}\rho
_{0}=0,\quad\rho_{1}=\left(  x+\omega\right)  \rho.
\]
If $\rho_{0}\left(  x\right)  $ is the weight function corresponding to the
Kravchuk polynomials (\ref{Kravchouk}), we can use Proposition \ref{Prop I}
with $q=x+\omega$ and obtain%
\[
\Lambda_{2}\left(  x\right)  =\frac{x}{z}\left(  x+\omega-1\right)  -\left(
x+\omega\right)  \left(  x-N\right)  ,\quad z\neq0,
\]
were we have also used (\ref{Krav Pearson}). The functional $L_{1}$ is a
Christoffel transformation of $L_{0}$ satisfying%
\[
\nabla\left[  \left(  x+\omega\right)  \rho_{1}\right]  +\Lambda_{2}\rho
_{1}=0,
\]
and using (\ref{Gen Krav Pearson}) we see that $\rho_{1}\left(  x\right)  $ is
the weight function corresponding to the generalized Kravchuk polynomials
(\ref{Gen Krav}) with%
\[
a=\omega,\quad b=\omega-1.
\]

\subsection{Case III: $\deg\left(  \Lambda_{3}\right)  =2$}

If we take $\Lambda_{3}=\left(  x+\omega\right)  \left(  x+a\right)  ,$
$\omega\in\mathbb{C},$ then from (\ref{r0, r1}) it follows%
\[
\nabla\left[  \left(  x+\omega\right)  \left(  x+a\right)  \rho_{0}\right]
+\Lambda_{2}\rho_{0}=0,\quad\rho_{1}=\left(  x+\omega\right)  \left(
x+a\right)  \rho_{0}.
\]
If $\rho_{0}\left(  x\right)  $ is the weight function corresponding to the
Meixner polynomials (\ref{Meixner}), we can use Proposition \ref{Prop I} with
$q=x+\omega$ and obtain%
\[
\Lambda_{2}\left(  x\right)  =\frac{x}{z}\left(  x+\omega-1\right)  -\left(
x+\omega\right)  \left(  x+a\right)  ,\quad z\neq0,
\]
were we have also used (\ref{M Pearson}). The functional $L_{1}$ is a
Christoffel transformation of $L_{0}$ satisfying%
\[
\nabla\left[  \left(  x+\omega\right)  \left(  x+a\right)  \rho_{1}\right]
+\Lambda_{2}\rho_{1}=0,
\]
and using (\ref{G Hahn Pearson}) we see that $\rho_{1}\left(  x\right)  $ is
the weight function corresponding to the generalized Hahn polynomials of type
I (\ref{G Hahn}) with%
\[
a_{1}=a,\quad a_{2}=\omega,\quad b=\omega-1.
\]

\subsection{Case IV: $\deg\left(  \Lambda_{3}\right)  =3$}

If we take $\Lambda_{3}=\left(  x+\omega\right)  \left(  x-N\right)  \left(
x+a\right)  ,$ $\omega\in\mathbb{C},$ $N\in\mathbb{N},$ from (\ref{r0, r1}) we
deduce that%
\begin{gather*}
\nabla\left[  \left(  x+\omega\right)  \left(  x-N\right)  \left(  x+a\right)
\rho_{0}\right]  +\Lambda_{2}\rho_{0}=0,\\
\rho_{1}=\left(  x+\omega\right)  \left(  x-N\right)  \left(  x+a\right)
\rho_{0}.
\end{gather*}
If $\rho_{0}\left(  x\right)  $ is the weight function corresponding to the
Hahn polynomials (\ref{Hahn}), we can use Proposition \ref{Prop I} with
$q=x+\omega$ and obtain%
\[
\Lambda_{2}=\allowbreak\left(  N-a+b-1\right)  x^{2}+\allowbreak\left(
N\omega+Na-a\omega+b\omega-b\right)  x+Na\omega,
\]
were we have also used (\ref{Hahn Pearson}). The functional $L_{1}$ is a
Christoffel transformation of $L_{0}$ satisfying%
\[
\nabla\left[  \left(  x+\omega\right)  \left(  x-N\right)  \left(  x+a\right)
\rho_{1}\right]  +\Lambda_{2}\rho_{1}=0,
\]
and using (\ref{G Hahn II Pearson}) we see that $\rho_{1}\left(  x\right)  $
is the weight function corresponding to the generalized Hahn polynomials of
type II (\ref{G Hanh II}) with%
\[
a_{1}=a,\quad a_{2}=\omega,\quad b_{1}=b,\quad b_{2}=\omega-1.
\]

Note that Christoffel transforms of $\Delta$-classical linear functionals have
been analyzed in \cite{MR1939588}.

\section{Conclusions and future directions}

We have classified all $\Delta$-coherent pairs of the second kind
$\{L_{0},L_{1}\}$ and derived the following results:%

\[%
\begin{tabular}
[c]{|l|l|}\hline
$L_{0}$ & $L_{1}$\\\hline
generalized Charlier & generalized Charlier\\\hline
Charlier & generalized Meixner\\\hline
Kravchuk & generalized Kravchuk\\\hline
Meixner & generalized Hahn of type I\\\hline
Hahn & generalized Hahn of type II\\\hline
\end{tabular}
.
\]
In all cases, the functional $L_{1}$ is a Christoffel transformation of
$L_{0},$ in agreement with the general results obtained by Meijer
\cite{MR1451509} for the continuous case.\newline

In a work in progress, we are dealing with analytic properties of discrete
Sobolev inner products associated with a pair of coherent pairs of the second
kind. We must point out that in \cite{MR2127867} the authors deal with
asymptotic properties and the location of zeros of discrete polynomials
associated with a Sobolev inner product
\[
\langle p,q\rangle_{\mathfrak{S}}=\langle u_{0},pq\rangle+\lambda\langle
u_{1},\Delta p\Delta q\rangle,
\]
where $\lambda\geq0$, $(u_{0},u_{1})$ is a $\Delta$-coherent pair of
positive-definite linear functionals, and $u_{1}$ is the Meixner linear
functional. A limit relation between these orthogonal polynomials and the
Laguerre-Sobolev orthogonal polynomials which is analogous to the one existing
between Meixner and Laguerre polynomials in the Askey scheme is deduced.
Notice that Meixner polynomials are $\Delta$ self-coherent, Thus, Mehler-Heine
type formulas and zeros of such polynomials when $u_{0}=u_{1}$ is the Meixner
functional have been studied in \cite{MR3319506}. Outer relative asymptotics
and Plancherel-Rotach asymptotics as well as limit relations are analyzed in
\cite{MR1741786}. Algebraic properties of such polynomials as well as the
behavior of their zeros appear in \cite{MR1752153}.

\section{Acknowledgements}

Diego Dominici would like to thank the Isaac Newton Institute for Mathematical
Sciences, Cambridge, for support and hospitality during the programme
"Applicable resurgent asymptotics: towards a universal theory" (ARA2). The
work was supported by EPSRC grant no EP/R014604/1.

Francisco Marcell\'an has been supported by FEDER/Ministerio de Ciencia e
Innovaci\'on-Agencia Estatal de Investigaci\'on of Spain, grant
PID2021-122154NB-I00, and the Madrid Government (Comunidad de Madrid-Spain)
under the Multiannual Agreement with UC3M in the line of Excellence of
University Professors, grant EPUC3M23 in the context of the V PRICIT (Regional
Program of Research and Technological Innovation).

\newif\ifabfull\abfullfalse\input apreambl

\end{document}